\documentclass[11pt]{amsart}
\pdfoutput=1
\usepackage[
paper=a4paper,
text={147mm,230mm},centering
]{geometry}
\usepackage{dsfont}
\iftrue
\makeatletter
\def\@settitle{%
  \vspace*{-20pt}
  \begin{flushleft}%
    \baselineskip14\p@\relax
    \normalfont\bfseries\LARGE
    \@title
  \end{flushleft}%
}
\def\@setauthors{%
  \begingroup
  \def\thanks{\protect\thanks@warning}%
  \trivlist
  \large \@topsep30\p@\relax
  \advance\@topsep by -\baselineskip
  \item\relax
  \author@andify\authors
  \def\\{\protect\linebreak}%
  \authors
  \ifx\@empty\contribs
  \else
    ,\penalty-3 \space \@setcontribs
    \@closetoccontribs
  \fi
  \normalfont
  \@setaddresses
  \endtrivlist
  \endgroup
}
\def\@setaddresses{\par
  \nobreak \begingroup\raggedright
  \small
  \def\author##1{\nobreak\addvspace\smallskipamount}%
  \def\\{\unskip, \ignorespaces}%
  \interlinepenalty\@M
  \def\address##1##2{\begingroup
    \par\addvspace\bigskipamount\noindent
    \@ifnotempty{##1}{(\ignorespaces##1\unskip) }%
    {\ignorespaces##2}\par\endgroup}%
  \def\curraddr##1##2{\begingroup
    \@ifnotempty{##2}{\nobreak\noindent\curraddrname
      \@ifnotempty{##1}{, \ignorespaces##1\unskip}\/:\space
      ##2\par}\endgroup}%
  \def\email##1##2{\begingroup
    \@ifnotempty{##2}{\smallskip\nobreak\noindent E-mail address%
      \@ifnotempty{##1}{, \ignorespaces##1\unskip}\/:\space
      \ttfamily##2\par}\endgroup}%
  \def\urladdr##1##2{\begingroup
    \def~{\char`\~}%
    \@ifnotempty{##2}{\nobreak\noindent\urladdrname
      \@ifnotempty{##1}{, \ignorespaces##1\unskip}\/:\space
      \ttfamily##2\par}\endgroup}%
  \addresses
  \endgroup
  \global\let\addresses=\@empty
}
\def\@setabstracta{%
    \ifvoid\abstractbox
  \else
    \skip@25\p@ \advance\skip@-\lastskip
    \advance\skip@-\baselineskip \vskip\skip@
    \box\abstractbox
    \prevdepth\z@ 
    \vskip-10pt
  \fi
}
\renewenvironment{abstract}{%
  \ifx\maketitle\relax
    \ClassWarning{\@classname}{Abstract should precede
      \protect\maketitle\space in AMS document classes; reported}%
  \fi
  \global\setbox\abstractbox=\vtop \bgroup
    \normalfont\small
    \list{}{\labelwidth\z@
      \leftmargin0pc \rightmargin\leftmargin
      \listparindent\normalparindent \itemindent\z@
      \parsep\z@ \@plus\p@
      
    }%
    \item[\hskip\labelsep\bfseries\abstractname.]%
}{%
  \endlist\egroup
  \ifx\@setabstract\relax \@setabstracta \fi
}
\def\section{\@startsection{section}{1}%
  \z@{-1.2\linespacing\@plus-.5\linespacing}{.8\linespacing}%
  {\normalfont\bfseries\large}}
\def\subsection{\@startsection{subsection}{2}%
  \z@{-.8\linespacing\@plus-.3\linespacing}{.3\linespacing\@plus.2\linespacing}%
  {\normalfont\bfseries}}
\def\subsubsection{\@startsection{subsubsection}{3}%
  \z@{.7\linespacing\@plus.1\linespacing}{-1.5ex}%
  {\normalfont\itshape}}
\def\@secnumfont{\bfseries}
\makeatother
\fi 

\usepackage{amssymb}
\usepackage[all]{xy}
\usepackage{graphicx,color,float}
\usepackage[bookmarks]{hyperref}
\usepackage[textwidth=1.3in,color=yellow]{todonotes}
\usepackage{amsmath}
\usepackage{pb-diagram,pb-xy}
\usepackage{url}
\usepackage{tikz}

\textwidth=\paperwidth \advance\textwidth by-2\oddsidemargin
\advance\oddsidemargin by-1in
\evensidemargin=\oddsidemargin
\calclayout

\def\Z{\mathbb{Z}}

\def\Q{\mathbb{Q}}

\def\d{\partial}

\def\k{\mathds{k}}

\def\+{\oplus}
\def\hat{\widehat}
\newcommand\goesupto{\nearrow}
\newcommand\goesdownto{\searrow}

\theoremstyle{plain}
\newtheorem{theorem}{Theorem}[section]
\newtheorem{proposition}[theorem]{Proposition}

\newtheorem{lemma}[theorem]{Lemma}
\newtheorem*{theoremA}{Theorem A}
\newtheorem*{theoremB}{Theorem B}
\newtheorem*{theoremC}{Theorem C}

\newtheorem{theoremalpha}{Theorem}
\newtheorem{corollaryalpha}[theoremalpha]{Corollary} 
\theoremstyle{definition}
\newtheorem{definition}[theorem]{Definition}
\newtheorem{example}[theorem]{Example}
\newtheorem{remark}[theorem]{Remark}
\newtheorem{question}{Question}

\newtheorem*{acknowledgement}{Acknowledgement}
\def\to{\mathchoice{\longrightarrow}{\rightarrow}{\rightarrow}{\rightarrow}}
\makeatletter
\newcommand{\shortxra}[2][]{\ext@arrow 0359\rightarrowfill@{#1}{#2}}
\def\longrightarrowfill@{\arrowfill@\relbar\relbar\longrightarrow}
\newcommand{\longxra}[2][]{\ext@arrow 0359\longrightarrowfill@{#1}{#2}}

\makeatother

\newcommand{\cT}{\mathcal{T}}
\newcommand{\cC}{\mathcal{C}}
\newcommand{\spinct}{\mathfrak{t}}

\begin{document}

\title [An infinite-rank summand of knots with trivial Alexander polynomial]
{An infinite-rank summand of knots with trivial Alexander polynomial}

\author{Min Hoon Kim}
\address{
  School of Mathematics\\
  Korea Institute for Advanced Study \\
  Seoul 02455\\
  Republic of Korea
}
\email{kminhoon@kias.re.kr}

\author{Kyungbae Park}
\address{
	School of Mathematics\\ 
	Korea Institute for Advanced Study\\ 
	Seoul 02455\\
	Republic of Korea}
\email{kbpark@kias.re.kr}
\urladdr{newton.kias.re.kr/~kbpark} 



\maketitle
\begin{abstract}
	We show that there exists a $\Z^\infty$-summand in the subgroup of the knot concordance group generated by knots with trivial Alexander polynomial. To this end we use the invariant Upsilon $\Upsilon$ recently introduced by Ozsv\'ath, Stipsicz and Szab\'o using knot Floer homology. We partially compute $\Upsilon$ of $(n,1)$-cable of the Whitehead double of the trefoil knot. For this computation of $\Upsilon$, we determine a sufficient condition for two satellite knots to have identical $\Upsilon$ for any pattern with nonzero winding number.
\end{abstract}

\section{Introduction}



The celebrated theorems on 4-dimensional topology due to Donaldson \cite{Donaldson:1983-1} and Freedman \cite{Freedman:1982-1} have an interesting consequence on the study of knot concordance. Freedman proved that knots with trivial Alexander polynomial are topologically slice \cite{Freedman:1982-2,Freedman-Quinn:1990-1,Garoufalidis-Teichner:2004-1}. Using the Donaldson's diagonalization theorem, Casson and Akbulut had noticed that there exists a knot with trivial Alexander polynomial, which is not smoothly slice. (Their results are unpublished. See the paper of Cochran and Gompf \cite{Cochran-Gompf:1988-1}.) Topologically slice knots which are not smoothly slice measure the subtle difference between topological and smooth category in the study of 4-dimensional topology. For example, it is well-known that a topologically slice knot which is not smoothly slice gives an exotic $\mathbb{R}^4$ \cite[Theorem 9.4.23]{Gompf-Stipsicz:1999-1}. 

Following these results, topologically slice knots (modulo smooth concordance) have become of great interest and been studied extensively. We would like to review some results related to this. We first define certain subgroups of the knot concordance group. 

Let $\mathcal{C}$ and $\mathcal{C}^{top}$ be the smooth and topological knot concordance group, respectively. Let $\mathcal{C}_{T}$ be the subgroup of $\mathcal{C}$ consisting of topologically slice knots. That is, $\mathcal{C}_T$ is the kernel of the natural map $\mathcal{C}\to \mathcal{C}^{top}$. Let $\mathcal{C}_\Delta$ be the subgroup of $\mathcal{C}$ generated by knots with trivial Alexander polynomial. Then $\mathcal{C}_\Delta$ is a subgroup of $\mathcal{C}_T$ by the aforementioned result of Freedman. 

Using Furuta's results \cite{Furuta:1990-1} on Fintushel-Stern invariants (defined in \cite{Fintushel-Stern:1985-1}), Endo \cite{Endo:1995-1} proved that there is a $\mathbb{Z}^\infty$-subgroup in $\mathcal{C}_\Delta$ consisting of certain Pretzel knots. In \cite{Hedden-Kirk:2011-1, Hedden-Kirk:2012-1}, Hedden and Kirk introduced similar gauge theoretic invariants and proved that there is a $\mathbb{Z}^\infty$-subgroup in $\mathcal{C}_\Delta$ consisting of the Whitehead doubles of certain torus knots. 

Using Ozsv\'{a}th-Szab\'{o}'s Heegaard Floer correction terms \cite{Ozsvath-Szabo:2003-2}, it is known that $\mathcal{C}_T/\mathcal{C}_\Delta$ is highly non-trivial. Hedden, Livingston and Ruberman \cite{Hedden-Livingston-Ruberman:2012-1} proved that $\mathcal{C}_T/\mathcal{C}_\Delta$ contains a $\Z^\infty$-subgroup, and Hedden, Kim and Livingston \cite{Hedden-Kim-Livingston:2016-1} showed that it also contains a $\Z_2^\infty$-subgroup. 

Despite of these developments, the splitting problem of $\mathcal{C}_T$ (for example, the existence of a $\Z^\infty$-summand in $\mathcal{C}_T$) was poorly understood and had been regarded as a difficult but interesting problem. The first result is due to Livingston. In \cite{Livingston:2004-1}, he found a $\Z$-summand in $\mathcal{C}_\Delta$ using Ozsv\'{a}th-Szab\'{o}'s $\tau$-invariant \cite{Ozsvath-Szabo:2004-1}. 
By the development of concordance homomorphisms to the integer, including the $\tau$-invariant, Ramsussen's $s$-invariant \cite{Rasmussen:2010-1} and Manolescu-Owens' $\delta$-invariant \cite{Manolescu-Owens:2007-1}, it is known that there exists a $\Z^3$-summand in $\mathcal{C}_\Delta$ \cite{Livingston:2008-1}. 
By generalizing $\delta$-invariant, Jabuka \cite{Jabuka:2012-1} obtained infinitely many knot concordance homomorphisms on $\mathcal{C}_T$. However, Jabuka's homomorphisms are very difficult to calculate simultaneously and linear independence of them is still mysterious.

Recently, using her $\epsilon$-invariant, Hom \cite{Hom:2015-3} proved the existence of a $\Z^\infty$-summand in $\mathcal{C}_T$. After Hom's work, Ozsv\'{a}th, Stipsicz and Szab\'{o} \cite{Ozsvath-Stipsicz-Szabo:2014-1} defined the Upsilon invariant, denoted by $\Upsilon$, and found a $\Z^\infty$-summand in $\mathcal{C}_T$ using $\Upsilon$. We will discuss later that Hom's knots and Ozsv\'{a}th-Stipsicz-Szab\'{o}'s knots have non-trivial Alexander polynomials. Now, it is natural to ask whether $\mathcal{C}_\Delta$ contains a $\Z^\infty$-summand or not. Our main result addresses this question. 

\begin{theoremalpha}\label{thm:A}There exists a $\Z^\infty$-summand in $\mathcal{C}_{\Delta}$.\end{theoremalpha}

Let $D$ be the positively-clasped untwisted Whitehead double of the right-handed trefoil knot, and $K_{p,q}$ be the $(p,q)$-cable of a knot $K$. To prove Theorem \ref{thm:A}, we consider $\{D_{n,1}\}_{n=2}^\infty$ and their Upsilon functions $\Upsilon_{D_{n,1}}$. Before we discuss the computation of $\Upsilon_{D_{n,1}}$, we compare our knots with the knots which are recently known to generate a $\Z^\infty$-summand of $\cC_{T}$. The followings are the sets of topologically slice knots considered by Hom \cite{Hom:2015-3} and Ozsv\'ath, Stipsicz and Szab\'o \cite{Ozsvath-Stipsicz-Szabo:2014-1} respectively: $$HOM:=\{D_{n,n+1}\#-T_{n,n+1}\}_{n=2}^\infty\text{ and }OSS:=\{D_{n,2n-1}\#-T_{n,2n-1}\}_{n=2}^\infty,$$ 
where $T_{p,q}$ is the $(p,q)$-torus knot and $-K$ denotes the mirror image of a knot $K$ with the opposite orientation. 
Observe that the Alexander polynomials of knots in $HOM$ and $OSS$ are non-trivial. Nonetheless, their knots might generate a $\Z^\infty$-summand in $\cC_\Delta$, but it seems hard to determine if they are not concordant to any knots with trivial Alexander polynomial. See \cite{Hedden-Livingston-Ruberman:2012-1} for a method to do such. We also point out that their knots are composite but ours are prime. Moreover, it will be shown in Section \ref{sec:ProofOfTheoremA} that the knots in $OSS$ are linearly independent to a subset of our knots, which also generate a $\Z^\infty$-summand in $\cC_{\Delta}$. 

We are back to discuss our computation of $\Upsilon_{D_{n,1}}$. Instead of computing $\Upsilon_{D_{n,1}}$ directly, we first obtain a condition for two satellite knots to have identical $\Upsilon$, using another concordance invariant from knot Floer homology, $\nu^+$ due to Hom and Wu \cite{Hom-Wu:2014-1}. We define two knots $K_1$ and $K_2$ to be \emph{$\nu^+$-equivalent} if $\nu^+(K_1\#-K_2)=\nu^+(-K_1\#K_2)=0$. It is known that $\nu^+$-equivalent knots have the same $\Upsilon$ functions \cite[Proposition 4.7]{Ozsvath-Stipsicz-Szabo:2014-1}. Here we develop further as follows: 
\begin{theoremalpha}\label{thm:B}Let $P$ be a pattern with nonzero winding number. If two knots $K_1$ and $K_2$ are $\nu^+$-equivalent, then $P(K_1)$ and $P(K_2)$ are $\nu^+$-equivalent, and consequently $\Upsilon_{P(K_1)}\equiv\Upsilon_{P(K_2)}$.
\end{theoremalpha}
 The proof of Theorem \ref{thm:B} hinges on the results of Cochran-Franklin-Hedden-Horn \cite{Cochran-Franklin-Hedden-Horn:2013-1}, Levine-Ruberman \cite{Levine-Ruberman:2014-1} and Ozsv\'{a}th-Szab\'{o} \cite{Ozsvath-Szabo:2003-2}. 
 
 In fact $D$ and $T_{2,3}$ are $\nu^+$-equivalent. See Example \ref{ex:D}. Therefore it is enough to compute  $\Upsilon_{T_{2,3;n,1}}$ for $\Upsilon_{D_{n,1}}$.  In Section \ref{sec:computation} we compute a part of $\Upsilon_{T_{2,3;n,1}}$ by understanding the infinity-version of the knot Floer chain complex of $T_{2,3;n,1}$ using the hat-version of knot Floer chain complex of $T_{2,3;n,1}$ given by Hom \cite{Hom:2015-1}.
 


\begin{theoremalpha}\label{thm:C} Let $T_{2,3;n,1}$ be the $(n,1)$-cable of the right-handed trefoil knot. Then 
	$$\Upsilon_{T_{2,3;n,1}}(t)=\left\{\begin{array}{lll}
	-nt&\text{if}&t\leq\frac{2}{1+n}\\
	t-2&\text{if}&\frac{2}{1+n}\leq t\leq\frac{2}{1+n}+\epsilon,
	\end{array}\right.$$
	for some small $\epsilon>0$. See Figure \ref{fig:Upsilon}.
\end{theoremalpha}

For \emph{thin} knots (see Definition \ref{def:thin}), we have the following Corollary of Theorem \ref{thm:B} (see Example \ref{ex:thin}).
\begin{corollaryalpha}\label{cor:D}Suppose $K_1$ and $K_2$ are thin knots. If $\tau(K_1)=\tau(K_2)$, then $\Upsilon_{P(K_1)}\equiv \Upsilon_{P(K_2)}$ for any pattern $P$ with non-zero winding number.
\end{corollaryalpha}

\begin{figure}
	\centering
	\begin{tikzpicture}[xscale=2.5, yscale=1.5]
	\draw[->] (0,0) -- (1.8,0) node[right] {$t$};
	\draw[->] (0,0) -- (0,-3.8) node[below] {$\Upsilon(t)$};
	
	\foreach \i in {2,3,4,5}   
	\draw[scale=2,domain=0:2/(1+\i),smooth,variable=\x,black] plot ({\x},{-\i*\x});
	
	\foreach \i in {2,3,4,5}  
	{\draw[scale=2,domain=2/(1+\i):2/(1+\i)+0.05,smooth,variable=\x,black] plot ({\x},{-2+\x});
		\draw[dotted, scale=2,domain=2/(1+\i):2/(1+\i)+0.1,smooth,variable=\x,black] plot ({\x},{-2+\x});}
	
	\draw[dashed] (4/5,-16/5)--(4/5,0) node [above] {\tiny{$\frac{2}{1+n}$}};
	\draw[dashed] (4/5,-16/5) -- (0,-16/5) node[left]  {\tiny{$\frac{-2n}{1+n}$}};
	
	\draw[dotted, thick] (1/2,-2.7)--(1/7,-2.7);
	
	\draw (1.5,-2.7) node {\tiny{$n=2$}};
	\draw (1.2,-3) node {\tiny{$n=3$}};
	\draw (1,-3.2) node {\tiny{$n=4$}};
	\draw (0.8,-3.4) node {\tiny{$n=5$}};
	
	\end{tikzpicture}
	\caption{The graphs of $\Upsilon_{T_{2,3;n,1}}\equiv\Upsilon_{D_{n,1}}$ for $n=2,\cdots,5$.}
	\label{fig:Upsilon}
\end{figure}
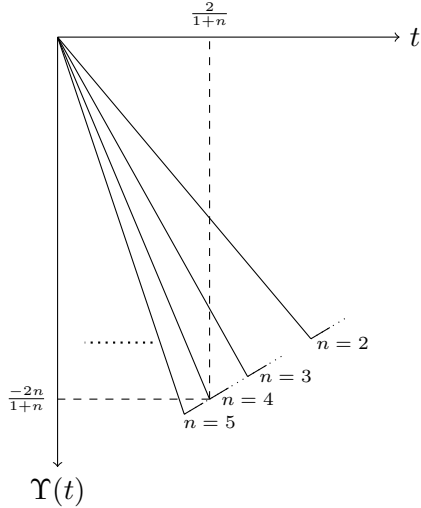
We conclude this section presenting questions naturally arisen during this work. 
\begin{question}
	Are  $\{D_{n,1}\}$ and the Hom's knots linearly independent in $\cC$?
\end{question}
As mentioned above, it follows from comparing $\Upsilon$ of our knots with Ozsv\'ath-Stipsicz-Szab\'o's that they are linearly independent in $\cC$. See the end of Section \ref{sec:ProofOfTheoremA}. However, by using Theorem \ref{thm:B} and a recent result of Wang \cite{Wang:2016-1}, it is easy to see that the first singularities of the $\Upsilon$ functions of the Hom's knots occur at the same $t$'s as ours. Therefore, the part of $\Upsilon$ we computed cannot determine the linear independence in $\cC$ of ours and Hom's.

\begin{question}Are the Hom's knots or the Ozsv\'ath-Stipsicz-Szab\'o's knots not concordant to any knots with trivial Alexander polynomial? Do they generate a $\Z^\infty$-summand in $\cC_T/\cC_\Delta$?
\end{question}
As long as the authors know, the only known method to answer the question is to use the obstruction of Hedden-Livingston-Ruberman \cite{Hedden-Livingston-Ruberman:2012-1} based on Heegaard Floer correction terms. Matt Hedden pointed out that, using the obstruction, one can show $D_{2,3}\#-T_{2,3}$ (Hom and Ozsv\'ath-Stipsicz-Szab\'o's knots for $n=2$) is not concordant to any knot with trivial Alexander polynomial. However it might need much more work to show those knots represent the nontrivial elements in $\cC_T/\cC_\Delta$.

\begin{question}\label{question:UpsilonAndCable}
	What is the behavior of $\Upsilon$ under the cabling operation? Is $\Upsilon_{K_{p,q}}$ determined by some invariants of $K$?
\end{question}
The behavior of $\tau$-invariant (the negative slope of $\Upsilon$ at $t=0$) under the cabling operation has been studied intensively. For example, see  \cite{Hedden:2009-1, VanCott:2010-1, Petkova:2013-1, Hom:2014-2}. In particular $\tau(K_{p,q})$ is determined by $\tau(K)$, $\epsilon(K)$, $p$ and $q$ \cite{Hom:2014-2}. Therefore one can expect analogous results for the $\Upsilon$-invariant. However the results for $\tau$ have relied on various computation techniques for the hat-version (simple version) of knot Floer homology. Even if we luckily compute a part of $\Upsilon$ for some specific knots using the hat-version of knot Floer homology, one might need better computation techniques of the full knot Floer chain complex, to answer this question.

\begin{remark}
	Preparing for this article, Chen informed us a partial answer for Question 3 \cite{Chen:2016-1}. In particular he obtained an inequality for $\Upsilon_{K_{p,q}}$ at $t\in[0,\frac{2}{p}]$ in terms of $\Upsilon_K$, $p$ and $q$, which is analogous to Van Cott's \cite{VanCott:2010-1} in the case of $\tau$.
\end{remark}

\begin{question}
	Is there a $\Z_2^\infty$-subgroup or a $\Z_2^\infty$-summand in $\cC_\Delta$?
\end{question}
As we mentioned above, it is known that there exists a $\Z_2^\infty$-subgroup in $\cC^T/\cC_{\Delta}$ \cite{Hedden-Kim-Livingston:2016-1} using Heegaard Floer correction terms.  Therefore, a natural question to ask is if there exists a $\Z_2^\infty$-subgroup in $\cC_\Delta$, or a $\Z_2^\infty$-summand in $\cC^T$, $\cC_\Delta$ or $\cC^T/\cC_{\Delta}$. However many concordance invariants such as $\tau$ and $\Upsilon$ do not work effectively to detect torsion elements in $\cC$.

\begin{acknowledgement}
We wish to thank Matt Hedden for providing comments on the earlier version of this paper and Wenzhao Chen for his interest in this article and sharing his partial answer for Question \ref{question:UpsilonAndCable}.
\end{acknowledgement}

\section{Background on Heegaard Floer homology}
In this section we briefly recall the Heegaard Floer homology for closed oriented 3-manifolds and the knot Floer homology for knots in $S^3$, and various smooth knot concordance invariants coming from them, and introduce some properties of those invariants. The purpose of this section is to set up notations and collect the background materials, which will be used later in the following sections. We refer the readers to a recent survey paper of Hom \cite{Hom:2015-2} for the more detailed expositions on this subject.

\subsection{Heegaard Floer homology and Knot Floer homology}
Let $Y$ be a closed oriented 3-manifold and $\spinct$ be a spin$^c$ structure over $Y$. Heegaard Floer homology, introduced by Ozsv\'ath and Szab\'o \cite{Ozsvath-Szabo:2004-1}, associates to $(Y,\spinct)$ a relatively graded chain complex $CF^\infty(Y,\spinct)$, a freely and finitely generated module over the Laurent polynomial $\mathbb{F}[U,U^{-1}]$, where $\mathbb{F}:=\Z/2\Z$. We call the homological grading of the chain complex the \emph{Maslov grading}, and the multiplication of $U$ lowers the grading by 2. 

From the definition of the boundary map, the negative power of $U$ naturally induces a filtration on the chain complex, called the \emph{algebraic filtration}. The filtered chain homotopy type of $CF^\infty(Y,\spinct)$ is an invariant of $(Y,\spinct)$. Hence, the homology of $CF^\infty(Y,\spinct)$, denoted by $HF^\infty(Y,\spinct)$, is an invariant of $(Y,\mathfrak{t})$, called \emph{Heegaard Floer homology} of $(Y,\mathfrak{t})$. The algebraic filtration of $CF^\infty$ allows us to define various versions of Heegaard Floer homology groups as follows: 
$$HF^-:=H_*(CF^\infty\{i<0\}), HF^+:=H_*(CF^\infty\{i\geq0\}) \text{ and }\hat{HF}:=H_*(CF^\infty\{i=0\}),$$
where  $CS$ denotes the sub or quotient complex of $C$ generated by the elements in the filtration levels in $S$. 

Ozsv\'ath and Szab\'o \cite{Ozsvath-Szabo:2004-2} and independently Rasmussen \cite{Rasmussen:2003-1}, discovered a knot $K$ in $S^3$ induces an additional filtration on a Heegaard Floer chain complex of $S^3$. We call the induced filtration the \emph{Alexander filtration}. Together with the algebraic filtration, the ${\Z\oplus\Z}$-filtered chain homotopy type of the chain complex is an invariant of $K$, denoted by $CFK^\infty(K)$ and called the \emph{full knot Floer chain complex}. The multiplication of $U$ lowers Alexander and Algebraic filtration by 1 in $CFK^\infty$. We usually use $i$ for the algebraic filtration level and $j$ for the Alexander filtration level of $CFK^\infty$, as describing a subset of the ${\Z\oplus\Z}$-filtration level. The homology of the associated graded chain complex of $CFK^\infty\{i=0\}$ is denoted by $\hat{HFK}(K)$, i.e.
$$\hat{HFK}(K):=\bigoplus_{k\in\Z}H_*(C\{i=0,j=k\}).$$
We define \emph{$\delta$-grading} by the Maslov grading subtracted by Alexander filtration.
\begin{definition}\label{def:thin}
	A knot $K\subset S^3$ is called \emph{thin} if $\hat{HFK}(K)$ is supported in a single $\delta$-grading.
\end{definition}

It is convenient to depict a knot Floer chain complex as dots and arrows in an $(i,j)$-plane, in which a dot at $(i,j)$-coordinate indicates an $\mathbb{F}$-generator of the chain complex at the filtration level $(i,j)$, and an arrow stands for a  non-trivial component of the ending dot in the differential of the starting dot. Note that an arrow starting from a dot in $(k,l)$-coordinate always maps to a dot in $CFK^\infty\{i\leq k,j\leq l\}$.

\subsection{Ozsv\'{a}th-Szab\'{o}'s correction terms}
For a closed oriented 3-manifold $Y$ with a torsion spin$^c$-structure $\spinct$ over $Y$, the relative Maslov grading of $CF^\infty(Y,\spinct)$ can be lifted to an absolute $\Q$-grading \cite{Ozsvath-Szabo:2003-2}.

Let 
$\cT^{+}:=\mathbb{F}[U,U^{-1}]/U\mathbb{F}[U]$. We usually call a free part in $HF^+(Y,\spinct)$ a $\cT^+$-tower. If $Y$ is a rational homology 3-sphere, it is known that the rank of $HF^+(Y,\spinct)$ as $\cT^{+}$-module is one. 
 We define the \emph{correction term}, or \emph{$d$-invariant}, of $(Y,\spinct)$ by the lowest absolute grading of the $\cT^+$-tower in $HF^+(Y,\spinct)$. 
 
If $H_1(Y;\Z)=\Z$ and $\spinct_0$ is the torsion $\textrm{spin}^c$-structure over $Y$, then it is known that the $\cT^+$-rank of $HF^
+(Y,\spinct_0)$ is two and the absolute grading of each $\cT^+$-tower is supported in $2\Z+1/2$ and  $2\Z-1/2$ respectively. We define $d_{\pm\frac{1}{2}}(Y,\spinct_0)$ to be the lowest absolute grading in each tower corresponding to the grading $2\Z\pm1/2$ respectively. 
\begin{example}It is known that $HF^+(S^3)=\mathcal{T}^+_{(0)}$, $\text{and }HF^+(S^2\times S^1,\spinct_0)=\cT^+_{(\frac{1}{2})}\oplus \cT^+_{(-\frac{1}{2})}$ \cite{Ozsvath-Szabo:2003-1},
where the numbers in the parenthesis indicate the lowest absolute Maslov grading of each tower $\cT^+$. Hence $d(S^3)=0$, and $d_{\pm\frac{1}{2}}(S^1\times S^2,\spinct_0)=\pm \frac{1}{2}$.
\end{example}

We recall a special case of \cite[Proposition 4.12]{Ozsvath-Szabo:2003-2} for the later purpose. For a knot $K\subset S^3$ and $n\in \Z$, let $S^3_{n}(K)$ be the 3-manifold obtained by the $n$-surgery of $S^3$ along $K$.
\begin{proposition}[{\cite[Proposition 4.12]{Ozsvath-Szabo:2003-2}}]\label{lem:Ozsvath-Szabo} For a knot $K\subset S^3$, $d_{\frac{1}{2}}(S^3_0(K),\spinct_0)=\tfrac{1}{2}+d(S^3_{1}(K))$.
\end{proposition}

The correction term provides a spin$^c$ rational homology cobordism invariant for spin$^c$ 3-manifolds with \emph{standard} $HF^\infty$ (which is true if $b_1\leq 2$). See \cite[Theorem 1.2]{Ozsvath-Szabo:2003-2} for the case of rational homology 3-spheres. For 3-manifolds with $H_1\cong\Z$, this can be obtained from \cite[Theorem 9.15]{Ozsvath-Szabo:2003-2}. In \cite{Levine-Ruberman:2014-1}, Levine and Ruberman generalize the correction terms to $d(Y,\spinct,V)$ for any 3-manifold $Y$ with standard $HF^\infty$, a torsion $\textrm{spin}^c$-structure $\spinct$ over $Y$ and  a subspace $V\subset H_1(Y;\Z)$. They showed the rational corbordism invariance of the generalized correction terms in \cite[Proposition 4.5]{Levine-Ruberman:2014-1}. 
\begin{remark}\label{rem:Levine-Ruberman}If $H_1(Y;\Z)=\Z$ and $\spinct_0$ is the torsion $\textrm{spin}^c$-structure over $Y$, then it is immediate from the definitions that $d_{\frac{1}{2}}(Y,\spinct_0)=d(Y,\spinct_0,0)$.
\end{remark}
\begin{proposition}[{\cite[Theorem 9.15]{Ozsvath-Szabo:2003-2}, \cite[Proposition 4.5]{Levine-Ruberman:2014-1}}]\label{lem:Levine-Ruberman}Suppose that $Y_i$ is a closed, oriented \textup{3}-manifold such that $H_1(Y_i;\Z)\cong\Z$ for $i=1,2$. Denote the torsion spin$^c$-structure on $Y_i$ by $\spinct_i$. Suppose that there exists a spin$^c$-rational homology cobordism $(W,\mathfrak{s})$ from $(Y_1,\spinct_1)$ to $(Y_2,\spinct_2)$. Then, $d_{\frac{1}{2}}(Y_1,\spinct_1)=d_{\frac{1}{2}}(Y_2,\spinct_2)$.
\end{proposition}
\begin{proof}By Remark \ref{rem:Levine-Ruberman}, this is a special case of \cite[Proposition 4.5]{Levine-Ruberman:2014-1} when $V=V'=0$.
\end{proof}
\subsection{Concordance invariants from the knot Floer homology}

Let $K$ be a knot in $S^3$ and $C$ be the full knot Floer chain complex of $K$, $CFK^\infty(K)$. The concordance invariants $\tau$, $\nu^+$ and $\Upsilon$ are induced from the natural maps between sub or quotient complexes of $C$. We use the coordinate $(i,j)$ again for the (algebraic, Alexander)-filtration level of $C$.

Consider the sequence of the inclusion maps between the quotient complexes of $C$:
$$\iota_k:C\{i=0,j\leq k\}\rightarrow C\{i=0\}.$$
Then $\tau(K)$ is defined by the minimum of $k$ such that $\iota_k$ induces a nontrivial map on the homology \cite{Ozsvath-Szabo:2003-1}.

The invariant $\nu^+$ is defined in a similar manner. This time we consider the following sequence of projection maps, 
$$v^+_k:C\{\max\{i,j-k\}\geq0\}\rightarrow C\{i\leq 0\}.$$
Then we define $\nu^+$ of $K$ by the minimum of $k$ such that $(v^+_k)_*(1)=1$, where $1$ denotes the lowest graded generator of the $\cT^+$-tower in the homology of each complex \cite{Hom-Wu:2014-1}.



Instead of the original definition of the \emph{Upsilon invariant}, denoted by $\Upsilon$, in \cite{Ozsvath-Stipsicz-Szabo:2014-1}, we use the alternative definition due to Livingston \cite{Livingston:2015}. Fix a real number $t$ in $[0,2]$. Consider the following 1-parameter family of inclusion maps 
$$d_s:C\{j\leq\tfrac{2}{t}{s}+(1-\tfrac{2}{t})i\}\rightarrow C.$$
The Upsilon invariant of $K$ at $t$, $\Upsilon_K(t)$, is defined by the minimum of $-2s$ such that the image of $(d_s)_*$ contains the nontrivial element of grading 0. The $\Upsilon$ maps a knot $K$ to a continuous piecewise-linear function on $[0,2]$ such that $\Upsilon_K(0)=\Upsilon_K(2)=0$ and $\Upsilon_K(t)=\Upsilon_K(2-t)$.

We recall some properties of $\tau$, $\nu^+$ and $\Upsilon$ that will be used in this paper.
\begin{proposition}\label{prop:HFinvariants} Let $K$ be a knot. The invariants $\tau$, $\nu^+$ and $\Upsilon$ satisfy the following properties\textup{:}
\begin{enumerate}
	\item \cite{Ozsvath-Szabo:2004-1, Hom-Wu:2014-1, Ozsvath-Stipsicz-Szabo:2014-1} The invariants $\tau$, $\nu^+$ and $\Upsilon$ are smooth concordance invariants. In particular, $\tau$ and $\Upsilon$ induce group homomorphisms from $\cC$ to $\Z$ and $Cont([0,2])$ respectively. 
	\item \cite[Proposition 4.7]{Ozsvath-Stipsicz-Szabo:2014-1} $|\Upsilon_K(t)|\leq t\max(\nu^+(K),\nu^+(-K))$ for any $t\in[0,2]$.
	\item \cite[Proposition 1.6]{Ozsvath-Stipsicz-Szabo:2014-1} $\Upsilon'_K(0)=-\tau(K)$.
	\item \cite[Proposition 2.3 (2)]{Hom-Wu:2014-1} $\nu^+(K)=0$ if and only if $d(S^3_1(K))=0$.
\end{enumerate}
\end{proposition}
Note that Proposition \ref{prop:HFinvariants} (4) is originally stated as $\nu^+(K)=0$ if and only if $V_0(K)=0$ in \cite{Hom-Wu:2014-1}. The invariant $V_0$ is another concordance invariant equivalent to $-\frac{1}{2}d(S^3_1(K))$.

\section{$\nu^+$-equivalence and proof of Theorem \ref{thm:B}}

We define two knots $K_1$ and $K_2$ to be \emph{$\nu^+$-equivalent} if $$\nu^+(K_1\#-K_2)=\nu^+(-K_1\#K_2)=0.$$ In this section, we give an equivalent condition of $\nu^+$-equivalence and examples of $\nu^+$-equivalent knots
, and prove Theorem \ref{thm:B}. The following lemma is implicit in \cite{Hom:2015-2}.
\begin{lemma}[{\cite[Proposition 3.11]{Hom:2015-2}}]\label{lem:Hom} Suppose $K_1$ and $K_2$ are two knots in $S^3$. Then, $K_1$ and $K_2$ are $\nu^+$-equivalent if and only if there is a filtered chain homotopy equivalence $CFK^\infty(K_1)\oplus A_1\simeq CFK^\infty(K_2)\oplus A_2$ for some acyclic complexes $A_1$ and $A_2$. 
\end{lemma}
\begin{proof}In \cite[Theorem 1]{Hom:2015-2}, Hom proved that if $K_1$ and $K_2$ are concordant, then $CFK^\infty(K_1)\oplus A_1\simeq CFK^\infty(K_2)\oplus A_2$ for some acyclic complexes $A_1$ and $A_2$. The same proof works under the weaker assumption that $K_1$ and $K_2$ are $\nu^+$-equivalent. In Hom's proof, the assumption (that $K_1$ and $K_2$ are concordant) is used just once to argue that $\nu^+(K_1\#-K_2)=0$. Compare \cite[Proposition 3.11]{Hom:2015-2}.

For the opposite direction, assume that $CFK^\infty(K_1)\oplus A_1\simeq CFK^\infty(K_2)\oplus A_2$ for some acyclic $A_1$ and $A_2$. By tensoring both sides by $CFK^\infty(-K_2)$, we obtain that
\[CFK^\infty(K_1\#-K_2)\oplus A_1'\simeq CFK^\infty(K_2\#-K_2)\oplus A_2'\simeq CFK^\infty(U)\oplus A_2'',
\]
where the last filtered chain homotopy equivlalence follows from \cite[Theorem 1]{Hom:2015-2}. 

As pointed out in \cite[Section 5.5]{Peters:2010-1}, $d(S^3_1(K))$ can be obtained  by the direct summand containing the generator of $H_*(CFK^\infty(K))$. 
Therefore, $d(S^3_1(K_1\#-K_2))=d(S^3_1(U))=0$. By Proposition \ref{prop:HFinvariants} (4), $\nu^+(K_1\#-K_2)=0$. The same argument shows that $\nu^+(-K_1\#K_2)=0$ and hence $K_1$ and $K_2$ are $\nu^+$-equivalent. This completes the proof.
\end{proof}

\begin{example}\label{ex:D}
	In \cite[Proposition 6.1]{Hedden-Kim-Livingston:2016-1}, it is shown that $CFK^\infty(D)\simeq CFK^\infty(T_{2,3})\oplus A$ for an acyclic summand $A$. Therefore, $D$ and $T_{2,3}$ are $\nu^+$-equivalent. Moreover, $\#_{i=1}^kD$ is $\nu^+$-equivalent to $T_{2,2k+1}$ by \cite[Theorem B.1]{Hedden-Kim-Livingston:2016-1}.
\end{example}

\begin{example}\label{ex:thin} Suppose that $K$ is a thin knot such that $\tau(K)=\pm k$ for $k\geq 0$.
By Petkova \cite[Section 3.1]{Petkova:2013-1} and Lemma \ref{lem:Hom}, $K$ is $\nu^+$-equivalent to $T_{2,\pm (2k+1)}$.
\end{example}

Note that it directly follows from Proposition \ref{prop:HFinvariants} (1) and (2) that $\Upsilon_{K_1}\equiv\Upsilon_{K_2}$ for two $\nu^+$-equivalent knots $K_1$ and $K_2$. We can develop further so that the $\Upsilon$ of some satellite knots of them are also identical. More precisely,
\begin{theoremB}Let $P$ be a pattern with nonzero winding number. If two knots $K_1$ and $K_2$ are $\nu^+$-equivalent, then $P(K_1)$ and $P(K_2)$ are $\nu^+$-equivalent, and consequently $\Upsilon_{P(K_1)}\equiv\Upsilon_{P(K_2)}$.
\end{theoremB}
We will first prove Theorem \ref{thm:B} assuming the following Lemma which hinges on the results of Cochran-Franklin-Hedden-Horn \cite{Cochran-Franklin-Hedden-Horn:2013-1} and Levine-Ruberman \cite{Levine-Ruberman:2014-1}. 
\begin{lemma}\label{Lemma:satellite-d} Suppose that $Q$ is a pattern with non-zero winding number such that $Q(U)$ is slice where $U$ is the unknot. Then, $d(S^3_1(Q(K))=d(S^3_1(K))$ for any knot $K$. 
\end{lemma}
\begin{proof}[Proof of Theorem \ref{thm:B}] Suppose that $P$ is a pattern with non-zero winding number and  $K_1$ and $K_2$ are $\nu^+$-equivalent knots. By Proposition \ref{prop:HFinvariants} (4), $K_1$ and $K_2$ are $\nu^+$-equivalent if and only if $d(S^3_1(K_1\#-K_2))=d(S^3_1(-K_1\#K_2))=0$.

As in \cite[Figure 5.2]{Cochran-Davis-Ray:2014-1}, let $Q$ be the pattern described in Figure \ref{fig:pattern}. (Here, $\overline{K_2}$ is the mirror image of $K_2$.) Note that $Q(U)=P(K_2)\#-P(K_2)$ is slice and 
\[Q(K_1\#-K_2)=P(K_1\#K_2\#-K_2)\#-P(K_2).\]
Since $K_2\#-K_2$ is slice, $Q(K_1\#-K_2)$ is concordant to $P(K_1)\#-P(K_2)$. (Recall that if $K$ and $K'$ are concordant, then $P(K)$ is concordant to $P(K')$ for any pattern $P$.) The winding number of $Q$ is equal to that of $P$ and hence non-zero. By applying Lemma \ref{Lemma:satellite-d} to $K=K_1\#-K_2$, we can conclude that 
\[ d(S^3_1(P(K_1)\#-P(K_2)))=d(S^3_1(K_1\#-K_2))=0.\]
The last equality follows from the assumption that $K_1$ and $K_2$ are $\nu^+$-equivalent and by Proposition \ref{prop:HFinvariants} (4). The same argument shows that $d(S^3_1(-P(K_1)\#P(K_2)))=0$. Therefore, $P(K_1)$ and $P(K_2)$ are $\nu^+$-equivalent. It follows from Proposition \ref{prop:HFinvariants} (2) that $\Upsilon_{P(K_1)}\equiv \Upsilon_{P(K_2)}$.
\end{proof}

\begin{figure}[htb]
     \centering
\includegraphics[scale=0.8]{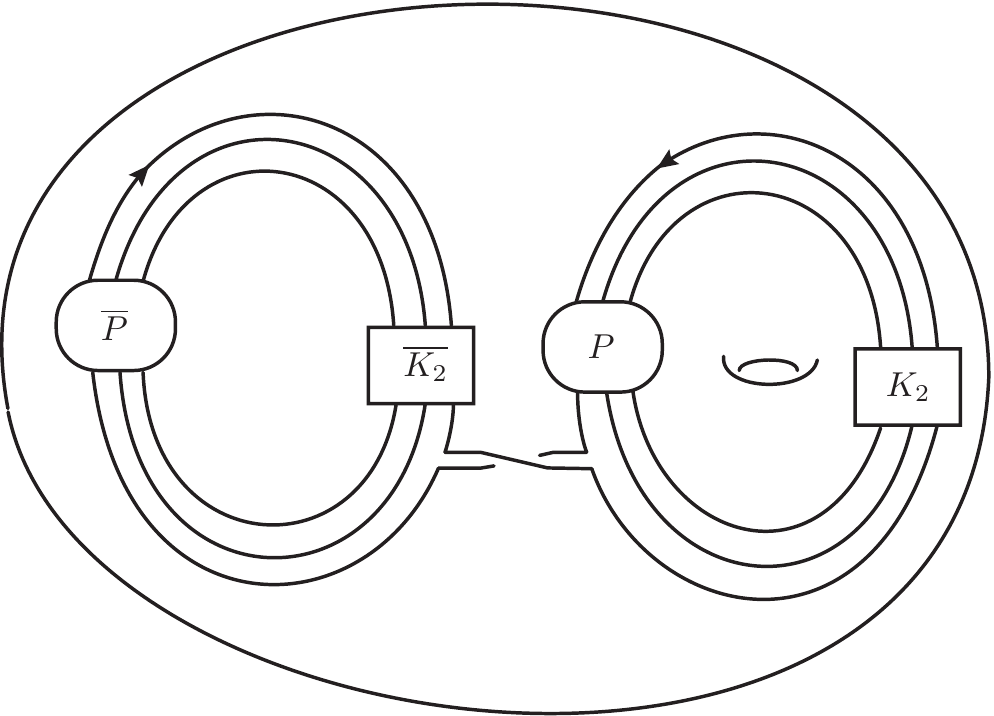}
\caption{The pattern $Q$ in the proof of Theorem \ref{thm:B}.}
\label{fig:pattern}
\end{figure}Now, we give a proof of Lemma \ref{Lemma:satellite-d}.
\begin{proof}[Proof of Lemma \ref{Lemma:satellite-d}] Suppose $Q$ is a pattern with non-zero winding number such that $Q(U)$ is slice. It follows from \cite[Theorem 2.1]{Cochran-Franklin-Hedden-Horn:2013-1} that $S^3_0(K)$ and $S^3_0(Q(K))$ are \emph{rationally homology cobordant}. That is, there exists a compact 4-manifold $W$ such that 
\begin{enumerate}
\item $\d W= S^3_0(Q(K))\sqcup - S^3_0(K)$.
\item $H_*(W;S^3_0(J);\Q)=0$ for $J=K$ and $Q(K)$.
\end{enumerate}
Since $W$ is a rational homology cobordism, the intersection form of $W$ is trivial and every second cohomology class in $W$ is characteristic. In particular, there is a $\operatorname{spin}^c$-structure $\mathfrak{s}$ on $W$ such that $c_1(\mathfrak{s})=0$. The restriction of $\mathfrak{s}$ to the boundary components are their unique torsion spin$^c$-structures. We can apply Propositions \ref{lem:Ozsvath-Szabo} and \ref{lem:Levine-Ruberman} to conclude that $d(S^3_1(K))=d(S^3_1(Q(K))$. This completes the proof. 
\end{proof}

\section{Computation of $\Upsilon_{T_{2,3;n,1}}$}
\label{sec:computation}
The goal of this section is to compute $\Upsilon_{T_{2,3;n,1}}(t)$ for some small $t$. Our strategy is to determine a part of $\Upsilon$ from some information about $CFK^\infty$, which can be obtained from $\hat{HFK}$. The starting point is $\hat{HFK}(T_{2,3;n,1})$ computed by Hom in \cite[Proposition 3.2]{Hom:2015-1}. We remark that the Hom's result is based on the Petkova's computation in \cite{Petkova:2013-1} using the bordered Floer homology. 

\begin{remark}
After the first draft of this paper, Wenzhao Chen informed us this computation can be alternatively obtained by using his inequality about $\Upsilon$ of cable knots \cite{Chen:2016-1}.
\end{remark}

As an $\mathbb{F}[U, U^{-1}]$-basis of $CFK^\infty(T_{2,3;n,1})$, we will take Hom's generators of $\hat{HFK}(T_{2,3;n,1})$ multiplied by  appropriate $U$-powers given in Table \ref{tab:cableinfty}. For the reader's convenience, we first recall the Hom's result on $\hat{HFK}(T_{2,3;n,1})$.
\begin{proposition}[{\cite[Proposition 3.2]{Hom:2015-1}}]
	\label{prop:cable}
	The group $\widehat{HFK}(T_{2,3;n,1})$ has rank $6n-5$. The generators are listed in Table \ref{tab:cable}, and the non-zero higher differentials are
	\begin{align*}
		\partial b_1 v_1 &= b_1\mu_1[n] \\
		\partial b_j v_1 &= b_{2n-j-1}v_1[n-j] &2 \leq j \leq n-1~\\
		\partial b_jv_2 &= b_{j+1}\mu_1[1]&1 \leq j \leq n-2~\\
		\partial b_{n-1}v_2 &= b_nv_2[1] \\
		\partial b_j \mu_2 &= b_{2n-j-1}\mu_2[n-j]&1\leq j \leq n-1,
	\end{align*}
	where the brackets denote the drop in Alexander filtration. For example, the Alexander filtration of $b_1 \mu_1$ is $n$ less than that of $b_1v_1$.
\end{proposition}
\begin{table}[htb!]
\vspace{5pt}
\begin{center}
\begin{tabular}{llllll}
\hline
&Generator \qquad \qquad & $(M, A)$ \qquad \qquad \qquad \qquad&&\\\hline
&$au_1$ & $(0, n)$  &\\
&$b_1v_1$  & $(-1, n-1)$  &\\
&$b_1\mu_1$ & $(-2, -1)$  &\\
&$b_jv_2$ & $(-2j-1, -j)$  &$1 \leq j \leq n-2$&\\
&$b_{j+1}\mu_1$ \qquad \quad & $(-2j-2, -j-1)$   &$1 \leq j \leq n-2$&\\
&$b_{n-1}v_2 $ & $(-2n+1, -n+1)$  &&\\
&$b_nv_2$ & $(-2n, -n)$ & \\
&$b_jv_1$ & $(-1, -j+n)$  &$2 \leq j \leq n-1$&\\
&$b_{2n-1-j}v_1$ & $(-2, 0)$  &$2 \leq j \leq n-1$&\\
&$b_j\mu_2$ & $(0, -j+n)$  &$1 \leq j \leq n-1$ &\\
&$b_{2n-1-j}\mu_2$ & $(-1, 0)$  &$1 \leq j \leq n-1$ &\\\hline
\end{tabular}
\end{center}
\caption{The generators of $\widehat{HFK}(T_{2,3;n,1})$.}
\label{tab:cable}
\end{table}

Recall that $CFK^\infty(K)$ is well-defined up to filtered chain homotopy equivalence. Using \cite[Lemma 4.5]{Rasmussen:2003-1}, we change $CFK^\infty(T_{2,3;n,1})$ by filtered chain homotopy equivalence so that $CFK^\infty(T_{2,3;n,1})\{i=0\}$ is equal to  $\hat{HFK}(T_{2,3;n,1})$. From now on, the isomorphism type of the underlying module of $CFK^\infty(T_{2,3;n,1})$ is fixed. Therefore, $CFK^\infty(T_{2,3;n,1})$ is a free $\mathbb{F}[U,U^{-1}]$-module generated by 
\[S=\{x_k, x_k', y_l, y_l', z_k, w_l\mid 2\leq k\leq n-1, 1\leq l\leq n-1\}.\]
Here, the elements of $S$ are given in Table \ref{tab:cableinfty}. They are appropriate $U$-translates of Hom's generators. It is easy to check the following (see Figure \ref{fig:complex}):

\begin{enumerate}
\item The Maslov gradings of $x_k$ and $x_k'$ are 1.
\item The Maslov gradings of $y_k$, $y_k'$, and $z_k$ are 0.
\item The Maslov gradings of $w_k$ are $-1$.
\item The (algebraic, Alexander) filtrations of $x_k$ and $y_k$ are $(1,k)$ and $(0,k)$, respectively.
\item The (algebraic, Alexander) filtrations of $x_k'$ and $y_k'$ are $(k,1)$ and $(k,0)$, respectively.
\item The (algebraic, Alexander) filtrations of $z_k$ and $w_k$ are $(1,1)$ and $(0,0)$, respectively.
\end{enumerate}

\begin{table}[htb!]
\vspace{5pt}
\begin{center}
\begin{tabular}{lllllll}
\hline
&Generator \qquad \qquad  & $(Alg,Alex,M)$ &&\\\hline
&$y_n=au_1$  & $(0, n,0)$ &\\
&$x_n=U^{-1}b_1v_1$ &$(1,n,1)$ & &\\
&$y_1'=U^{-1}b_1\mu_1$ & $(1,0,0)$ &\\
&$x_{k}'=U^{-k}b_{k-1}v_2$ & $(k,1,1)$  &$2 \leq k \leq n-1$&\\
&$y_{k}'=U^{-k}b_{k}\mu_1$ \qquad \quad & $(k,0,0)$   &$2 \leq k \leq n-1$&\\
&$x_p'=U^{-n}b_{n-1}v_2 $ & $(n,1,1)$  &&\\
&$y_p'=U^{-n}b_nv_2$ & $(n,0,0)$  & \\
&$x_{k}=U^{-1}b_{n-k+1}v_1$ & $(1,k,1)$  &$2 \leq k\leq n-1$&\\
&$z_{k}=U^{-1}b_{n+k-2}v_1$ &$(1,1,0)$ &$2 \leq k \leq n-1$&\\
&$y_{k}=b_{n-k}\mu_2$ & $(0,k,0)$  &$1 \leq k \leq n-1$ &\\
&$w_{k}=b_{n+1-k}\mu_2$ & $(0,0,-1)$  &$1 \leq k \leq n-1$ &\\\hline
\end{tabular}
\end{center}
\caption{The generators of $CFK^\infty(T_{2,3;n,1})$.}
\label{tab:cableinfty}
\end{table} 

\begin{figure}
	\centering
	\begin{tikzpicture}[xscale=1.2, yscale=1.2]
	\draw (-0.5,0)--(3,0) ;
	\draw [dotted] (3,0)--(4,0);
	\draw [->](4,0)--(7,0) node [right] {$Alg.$};
	\draw (0,-0.5)--(0,3) ;
	\draw [dotted] (0,3)--(0,4);
	\draw [->](0,4)--(0,7) node [above] {$Alex.$};

	
	\foreach \i/\label in {6/n,5/{n-1},2/2}
	{
		\draw [fill=black] (1,\i) circle (1.5pt) node (x\i) [anchor=west]  {\tiny{$x_{\label}$}};
		\draw [fill=black] (\i,1) circle (1.5pt) node [anchor=south] {\tiny{$x_{\label}'$}};
	}
	
	\foreach \i/\label in {6/n,5/{n-1},2/2,1/1}
	{
		\draw [fill=black] (0,\i) circle (1.5pt) node [anchor=east] {\tiny{$y_{\label}$}};
		\draw [fill=black] (\i,0) circle (1.5pt) node (y\i) [anchor=north] {\tiny{$y_{\label}'$}};
	}
	
	\foreach \i/\label in {-0.2/1,-0.1/{2},0.2/{n-1}}
	\draw [fill=black] (0-\i,0+\i) circle (1.5pt) node [anchor=north east] {\tiny{$w_{\label}$}};
	\draw [dotted] (-0.1,0.1) -- (0.1,-0.1);

	\foreach \i/\label in {-0.1/n-1,0.1/2}
	\draw [fill=black] (1-\i,1+\i) circle (1.5pt) node [anchor=north east] {\tiny{$z_{\label}$}};
	\draw [dotted] (1-0.1,1+0.1) -- (1+0.1,1-0.1);
	
	\foreach \i/\j/\k/\l in {1/6/0/6, 1/5/0/5, 1/2/0/2}  
	{
		\draw [->] (\j,\i-0.1)--(\l,\k+0.1);		
	}	
	
	\draw [->] (1-0.1,6)--(0+0.1,6);	
	
	\path 
	(1+0.05,6-0.05) edge [->, out=280, in=80] (1+0.05,0+0.05) 
	(1+0.05,5-0.05) edge [->, out=280, in=80] (1+0.1,1-0.1+0.05) 
	(1+0.05,2-0.05) edge [->, out=280, in=70] (1-0.1+0.05,1+0.1+0.05) 
	
	(-0.05, 5-0.05) edge [->, out=260, in=100] (0-0.05-0.2, 0+0.05+0.2) 
	(-0.05, 2-0.05) edge [->, out=260, in=100] (0-0.05+0.1, 0+0.05-0.1)
	(0, 1-0.05) edge [->, out=270, in=100] (0-0.05+0.2, 0+0.05-0.2);

	\draw [->, dashed] (1-0.05,6-0.05)--(0+0.05,2+0.05);
	\draw [->, dashed] (1-0.05,0)--(0+0.1+0.05,0-0.1);
	
	\draw [dotted] (1, 3)--(1,4);
	\draw [dotted] (3,1)--(4,1);
		\draw [dotted] (0, 3)--(0,4);
		\draw [dotted] (3,0)--(4,0);

	\draw [dashed, scale=1, domain=-0.2:1.1,smooth,variable=\x, blue] plot ({\x},{-5.5*\x+5.5});

	

	\end{tikzpicture}
	\caption{A summand $C_{model}$ of the complex $CFK^\infty(T_{2,3;n,1})$.} 
	\label{fig:complex}
\end{figure}
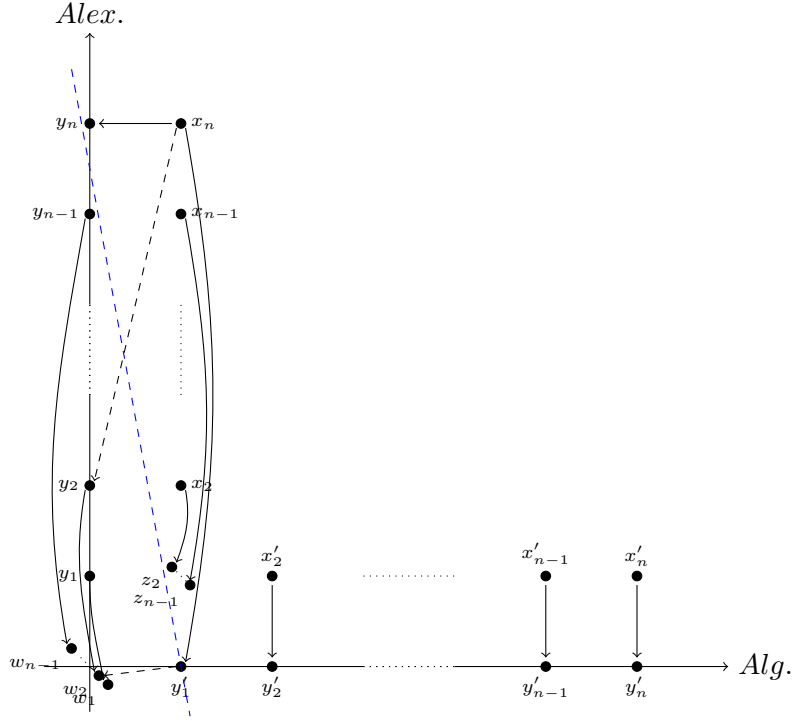

\begin{lemma}\label{lem:splitting}Let $C_{\mathrm{model}}$ be the $\mathbb{F}$-module generated by 
 \[S=\{x_k, x_k', y_l, y_l', z_k, w_l\mid 2\leq k\leq n-1, 1\leq l\leq n-1\}.\]
Then, $C_{\mathrm{model}}$ is a subcomplex of $CFK^\infty(T_{2,3;n,1})$. That is, $\d C_{\mathrm{model}}\subset C_{\mathrm{model}}$.
\end{lemma}
Before proving this lemma, we first discuss its consequence.
\begin{remark}As an $\mathbb{F}[U,U^{-1}]$-module, $CFK^\infty(T_{2,3;n,1})\cong C_{\mathrm{model}}\otimes \mathbb{F}[U,U^{-1}]$. Since the differential is $U$-equivariant, Lemma \ref{lem:splitting} gives an isomorphism between chain complexes as $CFK^\infty(T_{2,3;n,1})\cong C_{\mathrm{model}}\otimes\mathbb{F}[U,U^{-1}]$ where $\mathbb{F}[U,U^{-1}]$ is endowed with trivial differential and the tensor product is over $\mathbb{F}$. 
\end{remark}
As usual, we may write the differential as $\d=\d_V+\d_D+\d_H$. Here, $\d_V$, $\d_D$ and $\d_H$ are the vertical differential, the diagonal differential and the horizontal differential, respectively. As the terminologies suggest, $\d_V$, $\d_D$ and $\d_H$ are defined via the following conditions:
 \begin{enumerate}
 \item $\d_V$ fixes algebraic filtration but strictly lowers Alexander filtration.
 \item $\d_H$ fixes Alexander filtration but strictly lowers algebraic filtrations. 
 \item $\d_D$ strictly lowers both filtrations.
 \end{enumerate}
 Note that there is no differential component between the same $(i,j)$-filtration level since $CFK^\infty(T_{2,3;n,1})\{i=0\}=\hat{HFK}(T_{2,3;n,1})$.
 \begin{proof}[Proof of Lemma \ref{lem:splitting}] The lemma follows from the consideration of the Maslov grading $M$ and $\Z\oplus\Z$-filtration. Let $C^\infty:=CFK^\infty(T_{2,3;n,1})$. Recall that $C^\infty\{i=0\}=\hat{HFK}(T_{2,3;n,1})$. Since $U$ commutes with $\d$, Proposition \ref{prop:cable} determines $\d_V$ (see also Figure \ref{fig:complex}):
\[
\begin{array}{lll}
	&\partial_V x_n=y_1', &\\
	&\partial_V x_{k}= z_{k}  &\quad 2 \leq k \leq n-1, \\
	&\partial_V x_{k}'= y_{k}' &\quad 2 \leq k\leq n, \\
	\text{and }&\partial_V y_{k}= w_{k}  &\quad 1 \leq k \leq n-1.
\end{array}
\]
 We first observe that $\d y_k=\d_V y_k=w_k$ for all $k$. This is equivalent to $\d_Dy_k+\d_Hy_k=0$. Since $M(y_k)=0$, $M(\d_Dy_k+\d_Hy_k)=-1$. Since $\d_D$ and $\d_H$ strictly lower the algebraic filtration, $\d_Hy_k+\d_Dy_k\in C^\infty\{i\leq -1\}$. Since multiplying by $U$ lowers the Maslov grading by $2$, every element of $C^\infty\{i\leq -1\}$ has the Maslov grading less than or equal to $-2$. For every $k$, it follows that $\d_Dy_k+\d_Hy_k=0$, and $\d y_k=w_k\in C_\mathrm{model}$.
 
 Since $\d^2=0$, $\d w_k=\d^2 y_k=0$ for all $k$. Recall that we have determined $\d_V$. It follows that $\d_Vz_k=0$ for all $k$. Then, $\d z_k=\d_Dz_k+\d_H z_k\in C^\infty\{i\leq 0,j\leq 1\}$ and $M(\d z_k)=-1$. By the grading consideration, we can write 
 \[\d_Hz_k=\epsilon_k Ux_2,\quad \d_Dz_k=\textstyle\sum\limits_{l=1}^{n-1} a_{k,l}w_l\]
  where $\epsilon_k, a_{k,l}\in \{0,1\}$. On the other hand, note that $\d_V Ux_2=Uz_2$ and hence $\d U x_2\neq 0$. Since $\d w_l=0$ for all $l$,
  \[\d^2z_k=\textstyle\sum\limits_{l=1}^{n-1}a_{k,l}\d w_l+\epsilon_k \d Ux_2=\epsilon_k\d Ux_2.\]
   Since $\d^2=0$ and $\d Ux_2\neq 0$, it follows that $\epsilon_k=0$ and $\d z_k=\textstyle\sum\limits_{l=1}^{n-1} a_{k,l}w_l\in C_{\mathrm{model}}$ for all $k$.
 
 Now, we prove that $\d x_k, \d x_k'\in C_{\mathrm{model}}$.
 Since $M(x_k)=M(x_k')=1$, $M(\d x_k)=M(\d x_k')=0$. Since the multiplication by $U$ or $U^{-1}$ changes the Maslov grading by $2$, the Maslov grading 0 elements are linear combinations of $y_k$, $y_k'$ and $z_k$. This shows that $\d x_k, \d x_k'\in C_{\mathrm{model}}$.

Since $M(y_k')=0$, $M(\d y_k')=-1$. Note that $\d y_k'\in C\{j\leq 0\}$. By grading and filtration consideration, we can write 
\[\d y_k'=\textstyle\sum\limits_{l=2}^{k+1}b_{k,l}Ux_l'+\sum\limits_{l=1}^{n-1}c_{k,l}w_l,\] where $b_{k,l}, c_{k,l}\in \{0,1\}$. It should be understood that $b_{k,n+1}=0$. Since $\d w_l=0$ for all $l$,  
\begin{equation*}\tag{$\ast$} 0=\d^2 y_k'=\textstyle\sum\limits_{l=2}^{k+1}b_{k,l}U\d x_l'\
\end{equation*}
for all $k$. If $b_{k,l}=0$ for all $k,l$, then $\d y_k'\in C_{\mathrm{model}}$ and we are done. Suppose that there is $k$ such that $b_{k,l}\neq 0$ for some $l$. Fix such $k$ and let $m$ be the maximum among $l$'s such that $b_{k,l}=1$. Note that $U\d x_m'=Uy_m'+U\d_Dx_m'+U\d_Hx_m'$. Then, $(\ast)$ implies that
\[0=U\d x_m'+\textstyle\sum\limits_{l=2}^{m-1}b_{k,l}U\d x_l'=Uy_m'+U\d_Dx_m'+U\d_Hx_m'+\textstyle\sum\limits_{l=2}^{m-1} b_{k,l}U\d x_l'.\]
Recall that the algebraic filtration of $Uy_m'$ is $m-1$. On the other hand, the algebraic filtration of the remaining terms in the right hand side are less than or equal to $m-2$. This is a contradiction and completes the proof.
\end{proof}
From the lemma above, we may write $\d y_1'=\sum\limits_{i=1}^{n-1} a_iw_i$ for some $a_i\in\{0,1\}$.
\begin{lemma}\label{lem:nontrivial} 
	\[[y_n]=[y_1'+\textstyle\sum\limits_{i=1}^{n-1}a_iy_i]\neq0\textrm{~in~}H_*(C^\infty).\]
\end{lemma}
\begin{proof}
	Note that $M(y_n)=0$. Since $y_n$ is null-homologous in $C^\infty\{j=n\}$ and $x_n$ is the unique element $C^\infty\{j=n\}$ with grading $-1$, $\partial_Hx_n=y_n$. See \cite[Lemma 3.3]{Hom:2015-1}. Therefore, we may write as 
	$$\partial x_n=(\d_V+\d_D+\d_H)x_n=y_1'+\textstyle\sum\limits_{i=1}^{n-1}b_iy_i+y_n$$
	for some $b_i\in\{0,1\}$. Note that, by Lemma \ref{lem:splitting}, it is enough to consider the elements in $C_{model}$ to prove this Lemma.
	
    Write as $\partial y_1'=\sum\limits_{i=1}^{n-1} a_iw_i$. By considering $\d^2x_n=0$, we obtain that $a_i=b_i$ for all $i$:
    \[ 0=\d^2x_n=\d \left(y_1'+\textstyle\sum\limits_{i=1}^{n-1}b_iy_i+y_n\right)=\textstyle\sum\limits_{i=1}^{n-1}(a_i+b_i)w_i.\]
By considering $\d x_n$, it follows that $[y_n]=[y_1'+\sum\limits_{i=1}^{n-1}a_iy_i]$. 

We will prove that $[y_1'+\sum\limits_{i=1}^{n-1}a_iy_i]\neq0$ in $H_*(C^\infty)$ by chasing diagram on the $(i,j)$-plane of $C_{model}$ (see Figure \ref{fig:complex}). Suppose that $y_1'+\sum\limits_{i=1}^{n-1}a_iy_i$ is a boundary. By the grading consideration, we may write as
	\begin{equation*}\tag{$\ast\ast$}\partial\left(\textstyle\sum\limits_{i=2}^{n}(c_ix_i+c_i'x_i')\right)=y_1'+\textstyle\sum\limits_{i=1}^{n-1}a_iy_i,\end{equation*}
	for some $c_i, c_i'\in \{0,1\}$. We look at $y_1'$ on the RHS of $(\ast\ast)$. Observe that the arrows coming to $y_1'$ are a vertical arrow from $x_n$ and possible diagonal arrows from $x_k'$'s for some $k=2,\cdots,n$. It follows from the above equation that either $c_n\neq0$ or $c'_k\neq0$ for some $k$. However, $c_n=0$, since $\partial x_n$ contains $y_n$, to which no other arrow maps, and the RHS of $(\ast\ast)$ dose not have $y_n$. 
	
	Now, pick $l_1$ such that $c'_{l_1}\neq0$. Since $\partial _Vx_{l_1}'=y_{l_1}'$ and the RHS of $(\ast\ast)$ does not have $y_{l_1}'$, there must be a diagonal arrow coming from $x'_{l_2}$ to $y'_{l_1}$ and $c_{l_2}\neq0$. By continuing this argument, we have that $c'_n\neq 0$. However this is not possible since $\partial x_n'$ contains $y_n'$, to which no arrow comes and the RHS of $(\ast\ast)$ does not have $y_n'$. Therefore, $y_1'+\sum\limits_{i=1}^{n-1}a_iy_i$ is not a boundary in $C_{model}$, and hence in $C^\infty$.
\end{proof}
Now, we are ready to prove Theorem \ref{thm:C}. For the reader's convenience, we recall the statement of Theorem \ref{thm:C}. 
\begin{theoremC} Let $T_{2,3;n,1}$ be the $(n,1)$-cable of the right-handed trefoil knot. Then 
	$$\Upsilon_{T_{2,3;n,1}}(t)=\left\{\begin{array}{lll}
	-nt&\text{if}&t\leq\frac{2}{1+n}\\
	t-2&\text{if}&\frac{2}{1+n}\leq t\leq\frac{2}{1+n}+\epsilon,
	\end{array}\right.$$
	for some small $\epsilon>0$.
\end{theoremC}
\begin{proof} Recall that the $\Upsilon$ is determined by when an inclusion map from a subcomplex of $C^\infty$ to $C^\infty$ is nontrivial in homology at the Maslov grading 0. Hence it is enough to consider the direct-summand $C_{model}$ in which all Maslov grading 0 elements of $C^\infty$ locate. 
Let $L$ be a line on the $(i,j)$-plane of $C_{model}$ with negative slope, and $C_{model}(L)$ be the subcomplex of $C_{model}$ generated by elements in the lower half-plane to $L$ (including $L$).

We first consider lines $L$ with slope less than $-n$. For those lines, if the $y$-intercept of $L$ is less than $n$, $H_*(C_{model}(L))_{(0)}=0$. 
If the $y$-intercept of $L$ is $\geq n$, then $[y_n]$ become nontrivial in $H_*(C_{model}(L))_{(0)}$. By Lemma \ref{lem:nontrivial}, $[y_n]$ is non-trivial in $H_*(C_{model})$. Since the equation for $L$ is given as $Alex=\tfrac{2}{t}{s}+(1-\tfrac{2}{t})Alg$, we get $\Upsilon_{T_{2,3;n,1}}(t)=-nt$ for $t\leq\frac{2}{1+n}$. 

Similarly, suppose that the slope of a line $L$ is between $-n$ and $-n+\delta$ for some $0<\delta<1$. In this situation, if $x$-intercept is less than 1, then $H_*(C_{model}(L))_{(0)}=0$. If $x$-intercept is equal to $1$, then $H_*(C_{model}(L))_{(0)}$ is generated by $[y_1'+\sum\limits_{i=1}^{n-1} a_iy_i]$ which is non-trivial in $H_*(C_{model})$ by Lemma \ref{lem:nontrivial}. Hence $\Upsilon_{T_{2,3;n,1}}(t)=t-2$ for $\frac{2}{1+n}\leq t\leq\frac{2}{1+n-\delta}$.
\end{proof}

\section{Proof of Theorem \ref{thm:A}}\label{sec:ProofOfTheoremA}
\begin{theoremA}There exists a $\Z^\infty$-summand in $\mathcal{C}_{\Delta}$.\end{theoremA}
\begin{proof}
	Let $D_{n,1}$ be the $(n,1)$ cable of positively-clasped untwisted Whitehead double of the trefoil knot, $n\geq2$. It is easy to see that the Alexander polynomial of $D_{n,1}$ is trivial from the fact that the Alexander polynomial of any untwisted Whitehead double is trivial and the cabling formula for the Alexander polynomial.
	
Let $\Delta\Upsilon'_k(t_0):=\lim\limits_{t\goesdownto t_0} \Upsilon'_K(t)-\lim\limits_{t\goesupto t_0} \Upsilon'_K(t)$.	Consider a homomorphism $\phi:\cC\rightarrow\Z^\infty$ defined by 
	$$\phi([K])=\left(\tfrac{1}{1+k}\Delta\Upsilon'_K(\tfrac{2}{1+k})\right)^\infty_{k=2}.$$
	Note that $\phi$ takes $[K]$ to finitely many nonzero integers by \cite[Propostions 1.4 and 1.7]{Ozsvath-Stipsicz-Szabo:2014-1}. Since $D$ and $T_{2,3}$ are $\nu^+$-equivalent, it follows from Theorems \ref{thm:B} and \ref{thm:C} that $\phi([D_{n,1}])=(*,\cdots,*,1,0,0,\cdots)$, where 1 is $n-1$'s coordinate. Therefore, $\phi$ is surjective, and Theorem A follows.
\end{proof}

In the introduction we claim that our knots are linearly independent to the knots introduced by Ozsv\'ath, Stipsicz and Szab\'o, which generate a $\Z^\infty$-summand in $\cC_T$. We prove the claim here.
\begin{proposition}
	Consider the following sets of knots:
	$$KP:=\{D_{2n-1,1}\}_{n=2}^\infty\text{ and }OSS:=\{D_{n,2n-1}\#-T_{n,2n-1}\}_{n=2}^\infty.$$ 
	Each of $KP$ and $OSS$ generate a $\Z^\infty$-summand in $\cC_T$, but they are linearly independent in $\cC$.
\end{proposition}
\begin{proof}
By restricting the range of $\phi$ (in the proof of Theorem \ref{thm:A} above) only to odd coordinates, one can easily see that $KP$ generates $\Z^\infty$-summand in $\cC_T$. Recall that the $\Upsilon$ functions of knots in $OSS$ have the first singularities at $t=\frac{2}{2n-1}$ (see the proof of \cite[Theorem 1.20]{Ozsvath-Stipsicz-Szabo:2014-1}, but knots in $KP$ have them at $t=\frac{2}{2n}$. 
\end{proof}
\bibliographystyle{amsalpha}
\renewcommand{\MR}[1]{}
\bibliography{research}
\end{document}